\documentclass[12pt,leqno,twoside]{amsart}

\usepackage[latin1]{inputenc}
\usepackage[T1]{fontenc}
\usepackage[colorlinks=true, pdfstartview=FitV, linkcolor=blue, citecolor=blue, urlcolor=blue]{hyperref}
\usepackage{amstext,amsmath,amscd, bezier,indentfirst,amsthm,amsgen,enumerate, geometry}
\usepackage[all,knot,arc,import,poly]{xy}
\usepackage{amsfonts,color}
\usepackage{amssymb}
\usepackage{latexsym}
\usepackage{epsfig}
\usepackage{graphicx}
\usepackage{srcltx}
\usepackage{enumitem}
\usepackage{tikz}
\usepackage{color}

\theoremstyle{plain}
\newtheorem{theorem}{Theorem}[section]
\newtheorem{lemma}[theorem]{Lemma}
\newtheorem{proposition}[theorem]{Proposition}
\newtheorem{corollary}[theorem]{Corollary}

\theoremstyle{definition}
\newtheorem{definition}[theorem]{Definition}
\theoremstyle{remark}
\newtheorem{remark}[theorem]{\sc Remark}
\newtheorem{example}[theorem]{\sc Example}

\def\bC{{\mathbb C}}
\def\bK{{\mathbb K}}

\def\bR{{\mathbb R}}

\def\cB{{\mathcal B}}
\def\cC{{\mathcal C}}

\newcommand{\Jac}{{\mathrm{Jac}}}
\newcommand{\Sing}{{\mathrm{Sing}}}

\def\const.{{\mathrm const.}}

\def\m{{\setminus}}
\makeatletter
\@namedef{subjclassname@2020}{\textup{2020} Mathematics Subject Classification}
\makeatother

\begin{document}
\title[Atypical values and the Milnor set of real polynomials in two variables]
 {Atypical values and the Milnor set of real polynomials in two variables}
\author{\sc Gabriel.  E.  Monsalve}  

\address{Departamento de Matem\'{a}tica, Instituto de Ci\^{e}ncias Matem\'{a}ticas e de Computa\c{c}\~{a}o, Universidade de S\~{a}o Paulo - Campus S\~{a}o Carlos and Laboratoire Paul Painlev\'e, Universit\'e de Lille,
59655 Villeneuve d'Ascq, France }

\email{esteban.monsalve@usp.br}

\thanks{ This work was supported by the grant $\#$2019/24377-2 and $\#$2020/14111-2, S\~{a}o Paulo Research Foundation (FAPESP). 
The author thanks to Prof. Mihai Tib\u{a}r for his guidance in this topic. }

\subjclass[2020]{14D06, 14Q20, 58K05, 57R45, 14P10, 32S20, 14P25.}
\date{\today}

\keywords{}





\begin{abstract}
We give a new algorithmic method of detection of atypical values for 2-variables real polynomial functions with emphasis on the effectivity.
\end{abstract}

\maketitle

\setcounter{section}{0}

\section{Introduction}
Let $ f: \bK^2 \to \bK$ be a non-constant polynomial function where $ \bK = \bR$ or $ \bC$.  It is well known that  $ f$ is a $ \cC^\infty$-trivial fibration outside a finite set in $\bK$, see \cite{Th,V}.   The smallest set $ \cB_f \subset \bK $ where $ f$ is a $ \cC^{\infty}$-trivial fibration outside $ \cB_f$ is called \textit{bifurcation set of }$f$ or \textit{set of atypical values of $f$}. The set $ f(\text{Sing}f)$ of critical values of $f$ is contained in $ \cB_f$. However,  $\cB_f$ may contain regular values, as the classical example $ f(x,y)=x+x^2y$ shows, since $ \text{Sing}f = \emptyset $ and  $\cB_f = \{0\}$.  Therefore, the bifurcation set $ \cB_f$ is the union of the finite set of critical values of $f$ with the set of atypical values which are regular values.  The problem of finding $ \cB_f$ relies on the detection of regular atypical values. 
For a complex polynomial  ($ \bK = \bC$) the following characterization of atypical values was found  in \cite{S}, and later in \cite{HL}: a regular value $ \lambda$ is atypical if and only if $ \chi(f^{-1}(\lambda)) \neq \chi(f^{-1}(t))$, where $ f^{-1}(t)$ is a general fibre and $ \chi$ denotes the Euler characteristic.  In 1995 this characterization was extended  in \cite{P} and \cite{ST} to polynomials $ f: \bC^n  \to \bC$ with \textit{isolated singularities at infinity}. 

In the real setting,  the variation of the Euler characteristic of the fibres is not enough to characterize the atypical values.  In \cite{TZ} the authors proved that for a function $ f : X \to \bR$ which is the restriction of a polynomial function $ F: \bR^n \to \bR$ to a smooth non-compact algebraic surface $ X \subset \bR^n$,  the existence of an atypical value is related to two phenomena which may occur at infinity: \textit{vanishing} and \textit{splitting} at infinity.  The next aim was to give a method to detect the vanishing and splitting at infinity.  In \cite{DJT} the authors present an effective algorithm that detects vanishing and splitting at infinity of a polynomial in two real variables based on the localization of the behavior of fibres at certain points at infinity.  In what concerns the detection of atypical values, the authors in \cite{CP} use certain truncated parametrizations of semialgebraic affine curves without treating the effectivity aspect.

In this paper we present  a method of detection of atypical values without the localization of the behavior of the fibres at infinity.  We use the notion of \textit{clusters of Milnor} arcs in order to detect the phenomena of vanishing and splitting as defined in \cite{DJT}. This notion is inspired from \cite{HN} where the authors consider clusters of the unbounded components of the curve  of tangency between the fibres of a rational function and the levels of the distance function.  In \cite{CP} the authors use clusters of the unbounded components of the polar curve.  

The detection of regular atypical values in this paper has mainly two new aspects: 

\noindent
(i) we give an effective way of detecting a compact set in $ \bR^2 $ such that on its complement the clusters are well-defined.  In \cite{CP,HN} the existence of such compact set is theoretically proved with semi-algebraic properties but the effectivity is not treated.  We also remark that in \cite{CP} and \cite{HN} one uses the notion of clusters to detect \textit{cleaving} instead of splitting,

\noindent
(ii) Theorem \ref{Thm: injFunctAlpha} proves a  injective correspondence between clusters and connected components of fibres of $f$ outside an  effectively determined compact set contained in the compact set in (i).

The paper is organized as follows:
in Section \ref{sec:MilnorSet} the \textit{Milnor set} $M_a(f)$ (see Definition \ref{Def:MilSet}) is introduced and a characterization of the \textit{primitiveness} (see \cite{DJT}) of a polynomial $f$ with respect to the dimension of  $M_a(f)$ is presented in Proposition \ref{Prop:StrenDTLemm} and Corollary \ref{cor:PrimitEquivDimM(f)1}.  The main result of this section shows that the non-primitiveness with respect to $\rho_a$ occurs for at most one point $ a\in \bR^2 $.  This represents an improvement of \cite[Lemma 2.3]{DT} in the case of polynomials in two variables.  
Under the primitiveness assumption, the set of $ \rho$-\textit{nonregular} points $ \mu(M(f))$ of the Milnor set centred at the origin $ M(f)$ (see Definition \ref{Def: MM }) is used to detect effectively the compact components of  $M(f)$,  see Theorem \ref{theo:charcCompCompo}.  The  \textit{Milnor radius at infinity} is introduced and will play an important role in the effectiveness of the detection.  

In Section \ref{Milnor arcs at infinity} the \textit{Milnor arcs at infinity of $f$} are defined (see Definition \ref{Def:MilArcInft}) and the constant type of tangency between the fibres and circles along the Milnor arcs at infinity is proved.

In Section \ref{sec":Order&Clust} the definition of cluster from \cite{HN} is adapted to our setting.  In this paper the clusters are called $ \mu$-clusters in order to point out the difference with the polar clusters in \cite{CP}.

In Section \ref{sec:injectiveFunction} it is proved that there is  an injective function between $\mu$-clusters associated to a regular value (see Definition \ref{d:cluster}) and non-compact fibres of $f$ outside a disk of large enough radius. This is important for relating the existence of atypical values at infinity with $\mu$-clusters in Section \ref{sec:AtipcDetect}.   

Theorem \ref{The:Atyp and Cluster} is our main result,  it characterizes the regular atypical values in terms of the parity of the $\mu$-cluster. More precisely: a regular value is atypical if and only if there exists an odd $\mu$-cluster associated to it.  The proof of this theorem also shows that the existence of an odd  $\mu$-cluster associated to a regular value is equivalent to the existence of either vanishing or splitting at infinity.    

In Section \ref{sec:AlgotithAspec} we give the algorithmic aspects of the detection of atypical values with special emphasis on effectivity.  This process has several steps:  finding the Milnor radius,  finding the values associated to the Milnor arcs,  finding the $\mu$-clusters, and finally finding the atypical values. The algorithm is applied in the examples of Section \ref{sec:Examples}.  These examples show  intriguing phenomena: Example \ref{Ex1} shows that both vanishing and splitting at infinity may occur at a single atypical value.  Example \ref{Ex2} shows that clusters may exist at $\lambda \in \bR$ without $\lambda$ being an atypical value.

\section{Milnor set and primitive polynomials}\label{sec:MilnorSet}

 Let $f : \bR ^2 \rightarrow \bR $ be a  non-constant polynomial function.
Let $\rho_a : \bR^2 \rightarrow \bR _{\geq0}  $, $\rho_a(x,y)=(x-a_1)^2+(y-a_2)^2 $ be the square of the  Euclidean distance to $a:= (a_1 , a_2) \in \bR^2$. 
After \cite{T1, T2, ACT1, ACT2} etc, one defines:
\begin{definition}\label{Def:MilSet}
The \textit{Milnor set of  $f : \bR ^2 \rightarrow \bR$ relative to} $\rho_{a}$ is the set of $\rho_a$-\textit{nonregular} points of $f$, namely:
$$M_a(f) := \{ (x,y) \in \bR^2 \mid  \rho_a \not\pitchfork_{(x,y)} f  \}.$$
\end{definition}
Equivalently, the Milnor set $M_a(f)$ is  the zero locus defined by the determinant of the Jacobian  matrix of the mapping $F_a:= ( f, \rho_a ): \bR^2 \rightarrow \bR^2$, namely $\Jac F_a(x,y) =0$.  In particular, if $a= (0,0)$ we denote $ \rho_a $ by $ \rho$  and $M_a(f)$ by $M(f)$.

\medskip

It is well-known that there exists an open dense set $ \Omega_f \subset \bR^2$ such that $M_a(f)$ is a curve for every $a \in \Omega_f$. Indeed, this holds for a polynomial function $f$ in any number of variables, cf   \cite{T1, T2, ACT1, ACT2,DTT}.  We want here a more effective result in 2 variables, and some more information on $M_a(f)$.

\begin{definition}\cite{DJT}
We say that a polynomial $f$ is \textit{primitive relative to} $\rho_a$, if $f \neq P \circ \rho_a$ for every polynomial $ P$ in one variable. 
\end{definition}

The following proposition extends \cite[Lemma 2.3]{DT} for polynomials in two variables.
\begin{proposition}\label{Prop:StrenDTLemm}
Let $ f : \bR^2 \to \bR$ be a non-constant polynomial function. Then: 
\begin{enumerate}
 \rm \item \it  For  any  $b \in \bR^2$, $M_b(f)$ is unbounded and intersects all fibres of $f$. 
\rm \item \it There exists at most one point $a \in \bR^2$ such that $f$ is not primitive relative to  $\rho_a$.
\rm \item \it  If $f$ is primitive with respect to $ \rho_b$, then $\dim (M_{b}(f) \setminus \Sing f)=1$.   
\end{enumerate}
\end{proposition}
\begin{proof}
(a). If $ \Sing  f$ is unbounded, then $M_b(f)$ is so,  since $\Sing  f \subset M_b(f)$ by Definition \ref{Def:MilSet}. Otherwise, for  all large enough circles $C_b$ centred at $b$, the restriction of $f$ to $ C_b$ has at least a maximum and a minimum point. These points are among the points where some fibre is tangent to the circle $C_b$,  and thus these points belong to $M_b(f)$ by definition. This proves that $M_b(f)$ is unbounded.  

Let us prove that any fibre $f^{-1}(t)$ intersects $M_b(f)$. If $t$ is a critical value of $f$, then the statement holds since $\Sing f  \subset M_{b}(f)$. If $t$ is a regular value, then the Euclidean distance from the point $b$ to the fibre $f^{-1}(t)$ has a minimum and thus $f^{-1}(t)$ intersects $M_b(f)$ along such points of minimum. 

\smallskip
\noindent
(b).  If $\dim M_{a}(f) = 2$, then  we have $ f= P \circ \rho_a$ for some polynomial $ P $ of one variable by Lemma \ref{Lem:DimM(F)=1} below. For any $b = (b_1,b_2) \neq a = (a_1,a_2)$, the set  $M_{b}(f) $ is defined by the equation:
\[ \Jac  (P(\rho_a), \rho_b)
   = 4P^{\prime}(\rho_a)( (a_2-b_2)x + (b_1 - a_1)y + a_1 b_2 - a_2 b_1 ) = 0,
\]where $P^\prime$ denotes the derivative of $P$. Hence $M_{b}(f)$  is the union of a line with the set  $ \{P^{\prime}(\rho_a) =0 \} $, where the later is a union of finitely many circles centred at $a$ since $P^\prime(t)=0 $ has finitely many solutions for $t>0$. This shows that $ \dim  M_{b}(f) \leq 1$. Moreover, for any $b \neq a$, the set $\{(a_2-b_2)x + (b_1 - a_1)y + a_1 b_2 - a_2 b_1 =0 \}$ is a nonempty line in $\bR^2$ passing through $b$. This proves that  $ \dim  M_{b}(f) = 1$ and $M_b(f)$ is unbounded, for any $b \neq a$. 

\smallskip
\noindent
(c).
If $f$ is primitive with respect to $\rho_b$, then $  \dim M_b(f) < 2 $ by Lemma \ref{Lem:DimM(F)=1}.  Since $M_b(f)$ intersects all fibres of $f$ and since ${ \text{Im}}f$ is unbounded, it follows that $M_b(f)$ is unbounded. Moreover, since $\Sing  f$ is contained in finitely many fibres, it also follows that $M_b(f) \setminus \Sing  f$ is unbounded.  Therefore there exists an unbounded component of  $ M_{b} (f) $ intersecting all fibres, and thus $ \dim (M_b(f)\setminus \Sing f)=1$.  
\end{proof}

\smallskip
The following lemma characterize the primitiveness of a polynomial function $ f$ relative to $ \rho_a $.  When the center $a$ is the origin of $ \bR^2$,  Lemma \ref{Lem:DimM(F)=1} coincides with \cite[Remark 2.4]{DJT}.

\begin{lemma}\label{Lem:DimM(F)=1}
Let $ f : \bR^2 \to \bR$ be a polynomial function and let $ a =(a_1,a_2)  \in \bR^2$. The following conditions are equivalent:

\begin{enumerate}
\rm \item \it $\dim  M_a(f) = 2$,
\rm \item \it $  \Jac F_a = 0$, 
\rm \item \it $f$ is not primitive with respect to $\rho_a$.
\end{enumerate}
\end{lemma}

\begin{proof}

(a)$\Leftrightarrow$ (b) It follows from the fact that $M_a(f)$ is a proper variety of $\bR^2$ if, and only if $\dim M_a(f) <2 $ (see for instance \cite[page.10]{M1}).  

\smallskip
\noindent
(c)$\Rightarrow$ (b) If $ f = P \circ \rho_a$ for some polynomial in one variable $P$,  then by the chain rule: \[ \Jac F_a (x,y) =P^{\prime}(\rho_a(x,y))\left[(x-a_1)(y-a_2)-(x-a_1)(y-a_2)\right] =0 \] for every $(x,y) \in \bR^2$,  where $P^{\prime}$  denotes the derivative of $P$.   

\smallskip
\noindent
(b)$\Rightarrow$ (c)  Let us assume that $ \Jac F_a = 0$.  Consider the translation $ T : \bR^2 \to \bR^2$, $ T(x,y)= (x-a_1 , y- a_2)$ and the composition map $ F_a \circ T^{-1}   = (f \circ T ^{-1} , \rho_a \circ T^{-1}) = (f \circ T ^{-1} , \rho) $, where $ \rho$ is the square of the euclidean distance to the origin. By the chain rule $   \Jac  (F_{a} \circ T^{-1})  = \Jac (F_a) \Jac  (T^{-1}) = \Jac (F_a) = 0 $. Let us denote $g =f \circ T^{-1} $, and thus $g$ is polynomial.

We use the polar change of coordinate  $g(r , \theta)= g(r\cos \theta, r \sin \theta )$ where $ x=r \cos \theta$, $y=r \sin \theta$ and  by the chain rule we obtain the derivatives:

\[ \begin{cases} 
      g_r= g_x \cos \theta + g_y \sin \theta,\\
      g_\theta = g_x (-r\sin \theta)+ 
       g_y(r \cos \theta). 
   \end{cases}
\]

Since $\Jac (g,\rho) = 0$,  it follows that $g_{\theta}=0$. Then
the function $g( r \cos \theta,r \sin \theta) $ is a polynomial depending only on $r$.

If $g(x,y) = \sum_{i,j} a_{i,j}x^i y^j $, then $g(r\cos\theta,r \sin  \theta)=\sum_{i,j} a_{i,j} r^{i+j} \cos^i\theta \sin^j\theta$. Setting  $\theta= \frac{\pi}{4}$ and $\theta = \frac{5\pi}{4}$ we obtain that  $a_{i,j} =0$ for $i+j$ odd, and that in 
  $g(r  \cos \theta,r \sin \theta) =  \sum_{k}  b_k(\theta) r^{2k}$ the coefficients $b_k$ are independent of $\theta$.
  
 Then we get
\begin{equation*}
\begin{split}
g(x, y) = \sum_{k}  b_k \left(\sqrt{ x^2 + y^2}\right)^{2k}   =\sum_{k}  b_k ( x^2 + y^2)^k.  
\end{split}
\end{equation*}Hence $ g = P \circ \rho$ where $ P(t) :=\sum_{k}  b_k t^k $. Therefore   $f = P \circ \rho \circ T = P \circ \rho_{a}$ and this ends the proof.
\end{proof}

The following corollary characterize the primitiveness of a polynomial function with the dimension of its Milnor set.  Its proof follows from Lemma \ref{Lem:DimM(F)=1} and Proposition \ref{Prop:StrenDTLemm}.

\begin{corollary}\label{cor:PrimitEquivDimM(f)1}
Let $ f:\bR^2 \to \bR$ be a polynomial function. Then $ f$ is primitive with respect to $ \rho$ if and only if $ \normalfont\text{dim}M(f)=1$.
\end{corollary}

\

\subsection{Compact components of $M(f)$ and radius at infinity}
 In this section we find effectively the radius of a disk centred in origin such that $M(f)$ outside this disk is a union of finitely many connected components of dimension 1. It follows by Proposition \ref{Prop:StrenDTLemm} that for all $a \in \bR^2$, except for at most one point in $\bR^2$,  the set $ M_a(f)$ is a 1 dimensional real curve.  Then,  after an appropriate translation of coordinates we may and we will always assume that $f$ is primitive with respect to $\rho= x^2+y^2$.  

We denote by $ M(f)_{\normalfont\text{red}}$ the reduced structure of the curve $ M(f)=\{ \text{Jac}(f,\rho)=0 \}$, i.e.  $ M(f)_{\text{red}}$ is the zero locus $\{ g= g_1\cdots g_s=0\}$, where $g_1 , \ldots , g_s$ are all the polynomial irreducible factors in the decomposition of $ \text{Jac}(f,\rho)$ with $ g_i \neq g_j$ when $i \neq j$. 
\begin{definition}\label{Def: MM }
Let $ f : \bR^2 \to \bR $ be a polynomial function. We define the set
$$ \mu(M(f)):= \{ p \in M(f)\text{  } | \text{  } \rho \not\kern-0.3em\cap\kern-0.5em|\kern0.7em\kern-0.5em_{p} M( f)_{\normalfont\text{red}} \}.$$
\end{definition}
By definition, $\mu(M(f))$ is a real algebraic set, since it is defined by the equations $g (x,y) =0$ and $ y g_{x}(x,y)-x g_y (x,y)=0  $, where $g$ is the reduced polynomial defining $ M(f)_{\text{red}}$ and $ g_x, g_y$ denote its  partial derivatives with respect to the variables $x$ and $y$, respectively.

\medskip

\begin{proposition}\label{Prop:MMCompact}
Let $f:\bR^2 \to \bR$ be a primitive polynomial function. Then the set $ \mu(M(f))$ is the union of finitely many points and circles centred at the origin. In particular, $ \mu(M(f))$ is compact.
\end{proposition}

\begin{proof}

Let us consider a semialgebraic Whitney stratification on $M(f)_{\normalfont\text{red}}$. By the Tarski-Seidenberg principle, all the levels of the distance function $ \rho $, except for finitely many, are transversal to the strata of $M(f)_{\normalfont\text{red}}$.  Hence the restriction $ \rho_{|M(f)_{\normalfont\text{red}}}$ has finitely many critical values and thus  $\mu(M(f)) \subset \cup_{i=1}^{s}C_{\lambda_i}$, where $C_{\lambda_i}$ denotes a circle centred at the origin of radius $\lambda_i$.  By B\'{e}zout Theorem one concludes that $ \mu(M(f)) $ is a union of finitely many points and circles centered at the origin.
\end{proof}

\begin{remark}\label{r:M(f)-unbo-outside}Due to Proposition \ref{Prop:MMCompact},  the set $\mu(M(f))$ is contained in a disk centred at the origin with radius large enough.  The restriction of $\rho$ to a compact component of $M(f)$ has a point of maximum and, by definition, this point is in $\mu(M(f))$.  Altogether this implies that $M(f)$ outside a disk that contains $\mu(M(f))$ is the union of unbounded components.  
\end{remark}

Next proposition gives a better description of the circles in $\mu(M(f))$, see  Proposition \ref{Prop:MMCompact}. 

\begin{proposition}\label{Prop:DetecCircleInMM}
Let $f: \bR^2 \to \bR$ be a polynomial. If $\mu(M(f))$ contains a circle $ C_r =\{ x^2+y^2 -r^2 =0\}$ of radius $r>0$ centred at the origin, then:

\begin{enumerate}
\rm \item \it$ C_{r} \subset M(f) $,
\rm \item \it $C_{r}$ is contained in a single fibre of $f$,
\rm \item \it There exists $\lambda \in \bR$ such that $f(x,y)-\lambda = (x^2+y^2-r^2)h(x,y)$ and thus $M(f) = C_{r} \cup M(h)$.
\end{enumerate}
\end{proposition}

\begin{proof}

(a). Follows straightforward from the inclusion $ \mu(M(f)) \subset M(f)$.  

\smallskip
\noindent(b). Let $\alpha_r : \left[0, 2\pi \right] \to C_{r} $ be the parametric equation of $C_{r} $.  By (a), the circle $ C_{r} \subset M(f) $ and thus the gradient vector $\nabla f(\alpha_r(t)) $ is a multiple scalar of $ \alpha_r(t)$. Hence $\left\langle \nabla f(\alpha_r(t)), \alpha_{r}^{\prime}(t) \right\rangle=0 $ for every $ t \in \left[0, 2\pi \right] $.  This proves that the restriction $ f_{|C_{r}} $ is constant,  equivalently, $ C_{r}$ is contained in a single fibre of $f$.

\smallskip
\noindent(c).  It follows from (b) that there exists $\lambda \in \bR$ such that the restriction $( f-\lambda)_{|C_r} =0$. Since $ x^2+y^2-r^2$ is an irreducible real polynomial in two variables, and $C_r$ is a one-dimensional algebraic set,  it follows by  \cite[Theorem 4.5.1]{BCR} that $ f-\lambda = (x^2+y^2-r^2)h(x,y)$, for some real polynomial $ h $ in two variables.

On the other hand, $M(f)$ is defined by the zero-set of the polynomial function

\[
\begin{split} 
 \text{Jac}(f,\rho) = & x(2yh(x,y) + (x^2 + y^2 -r^2)h_{y}(x,y))\\
                                    &-y(2xh(x,y)+(x^2+y^2-r^2)h_x(x,y))\\
                                  =  & (x^2+y^2-r^2)(xh_y(x,y)-yh_{x}(x,y)).
                             \end{split}
\]This proves that $ M(f) = C_{r} \cup M(h)$. 
\end{proof}

From Proposition \ref{Prop:DetecCircleInMM} we have:

\begin{theorem}\label{theo:charcCompCompo}
Let $ f: \bR^2 \to \bR $ be a  polynomial function and let $ C_r = \{ x^2+y^2-r^2=0 \}$. The following conditions are equivalent: 
\begin{enumerate}
\rm\item \it $C_{r}$ is contained in $M(f)$,
\rm\item\it there exists $\lambda \in \bR$ such that $f(x,y)= (x^2+y^2-r^2)h(x,y) +\lambda$, 
\rm\item\it  $ C_r$ is contained in a connected component of the fibre of $f$.
 \end{enumerate}  
\end{theorem}

By Proposition \ref{Prop:MMCompact} the set $\mu(M(f))$ is a compact set. Then there is a non-empty subset $ \Lambda \subset \bR_{>0}  $ of positive real numbers $ r>0$  such that $ \mu(M(f)) $ is contained in the interior of the open disk $ D_r$ centred at the origin of radius $r$.  Let us fix $ R_0 := \inf \{r \text{ } | \text{ } r \in \Lambda\} $  and we call $ R_0$  the \textit{Milnor radius at infinity of} $ f$.

\begin{remark}\label{re:}
For a primitive polynomial $f$, the set  $M(f) \m \overline{D}_{R_0}$ is a disjoint union of finitely many 1-dimensional manifolds.
\end{remark}


\section{ Milnor arcs at infinity}\label{Milnor arcs at infinity}
   
In this section we will describe the behavior of the Milnor set of a primitive polynomial function outside a disk $D_{R_0}$ centred at the origin of radius equals the Milnor radius at infinity of $f$. 
   
\begin{definition}\label{Def:MilArcInft}
Let $ f: \bR^2 \to \bR $ be a primitive polynomial function. A connected component  $ \gamma$ of $ M(f) \setminus \overline{D}_{R_0} $ will be called a \textit{Milnor arc at infinity of $f$}  and we denote by  $ \mathfrak{M}_{\text{arc}}(f)$ the set of Milnor arcs at infinity of $ f$.
\end{definition}

\begin{proposition} \label{Prop:RLargEnough}
Let $ f : \bR^2 \to\bR$ be a primitive  polynomial function and let $ \gamma $ be a Milnor arc at infinity $f$.  Then
\begin{enumerate}
\rm \item \it the function $\rho$ restricted to $\gamma$ is strictly increasing when $\gamma$ tends to infinity, and

\rm \item \it  if $ \gamma \not\subset  \normalfont\Sing f  $, then the restriction $f_{|\gamma}$ is either strictly increasing, or strictly decreasing.      
\end{enumerate}
\end{proposition}
\begin{proof}
\noindent
(a).  It follows straightforward from Proposition \ref{Prop:MMCompact} and Definition \ref{Def:MilArcInft}.
 
\smallskip
\noindent(b).  Let  $ \alpha : \left]R_0,\infty \right[ \to \bR^2$ be a parametrization of $ \gamma  $, where $ R_0$ denotes the Milnor radius at infinity of $f$. By contradiction,  let us suppose that $ f_{|\gamma} $ is not monotone.  Hence there exists $ t_0 \in \left]  R_0,\infty \right[$ such that $f(\alpha(t_0))$ is a local extrema and thus $ \left\langle \nabla f(\alpha(t_0)) , \alpha^{\prime} (t_0)  \right\rangle = 0$ holds. On the other hand, since $ \gamma \not\subset  \text{Sing}f $, there exists  a non-zero $ \beta \in \bR $ such that $ \nabla f(\alpha(t_0)) = \beta \alpha(t_0)$,  and thus $ \left\langle \alpha(t_0), \alpha^{\prime}(t_0)  \right\rangle = 0 $.  Therefore $ \alpha(t_0) \in \mu(M(f)) $ which is a contradiction with Definition \ref{Def:MilArcInft}.
\end{proof}

\medskip

By definition, the Milnor set $M(f)$ is the set of points where the fibres of $f$ are not transverse to the level sets of the Euclidean distance function $\rho$,  and the Milnor arcs do not intersect $\Sing f$, see Definition \ref{Def:MilArcInft}.  For any point $q$ of a Milnor arc $\gamma$,  the fibre of $f$ passing through $q$ may be either \begin{enumerate}
\rm\item locally inside the disk $D$,
\rm\item locally outside $D$,
\rm\item a local half-branch inside $D$ and the other local half-branch outside $D$,
\end{enumerate}where $D= \{\rho(x,y)  \le \|q\| \}$ is the disk centred a the origin of radius $\|q\|$
 
We say that the fiber of $f$ at $q$ has a: $\rho$-\textit{maximum} type of tangency if situation (a) holds,  or $\rho$-\textit{minimum} type of tangency if situation (b) holds,  or $\rho$-\textit{inflectional} type of tangency if situation (c) holds.

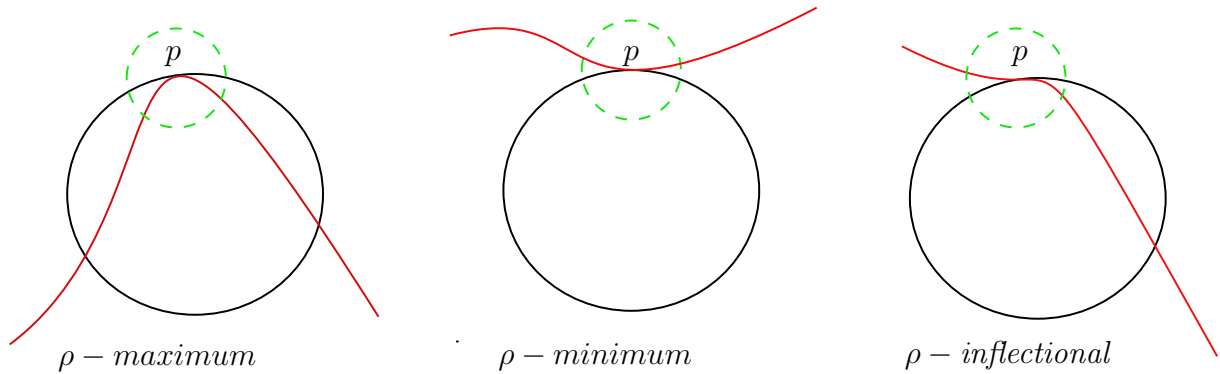
\tikzset{every picture/.style={line width=0.75pt}} 

\begin{tikzpicture}[x=0.75pt,y=0.75pt,yscale=-1,xscale=1]

\draw  [color={rgb, 255:red, 0; green, 0; blue, 0 }  ,draw opacity=1 ][line width=0.75] [line join = round][line cap = round] (270.5,230) .. controls (270.5,230) and (270.5,230) .. (270.5,230) ;
\draw   (74,155.75) .. controls (74,122.2) and (102.88,95) .. (138.5,95) .. controls (174.12,95) and (203,122.2) .. (203,155.75) .. controls (203,189.3) and (174.12,216.5) .. (138.5,216.5) .. controls (102.88,216.5) and (74,189.3) .. (74,155.75) -- cycle ;
\draw   (294,153.75) .. controls (294,120.2) and (322.88,93) .. (358.5,93) .. controls (394.12,93) and (423,120.2) .. (423,153.75) .. controls (423,187.3) and (394.12,214.5) .. (358.5,214.5) .. controls (322.88,214.5) and (294,187.3) .. (294,153.75) -- cycle ;
\draw   (499,157.75) .. controls (499,124.2) and (527.88,97) .. (563.5,97) .. controls (599.12,97) and (628,124.2) .. (628,157.75) .. controls (628,191.3) and (599.12,218.5) .. (563.5,218.5) .. controls (527.88,218.5) and (499,191.3) .. (499,157.75) -- cycle ;
\draw [color={rgb, 255:red, 208; green, 12; blue, 12 }  ,draw opacity=1 ]   (45,231.5) .. controls (144,157.5) and (78,-24.5) .. (231,217.5) ;
\draw  [color={rgb, 255:red, 16; green, 232; blue, 14 }  ,draw opacity=1 ][dash pattern={on 4.5pt off 4.5pt}][line width=0.75]  (104,97) .. controls (104,83.19) and (115.19,72) .. (129,72) .. controls (142.81,72) and (154,83.19) .. (154,97) .. controls (154,110.81) and (142.81,122) .. (129,122) .. controls (115.19,122) and (104,110.81) .. (104,97) -- cycle ;
\draw  [color={rgb, 255:red, 16; green, 232; blue, 14 }  ,draw opacity=1 ][dash pattern={on 4.5pt off 4.5pt}][line width=0.75]  (333.5,93) .. controls (333.5,79.19) and (344.69,68) .. (358.5,68) .. controls (372.31,68) and (383.5,79.19) .. (383.5,93) .. controls (383.5,106.81) and (372.31,118) .. (358.5,118) .. controls (344.69,118) and (333.5,106.81) .. (333.5,93) -- cycle ;
\draw [color={rgb, 255:red, 239; green, 11; blue, 11 }  ,draw opacity=1 ]   (267,75.5) .. controls (346,54.5) and (312,134.5) .. (452,61.5) ;
\draw [color={rgb, 255:red, 241; green, 17; blue, 17 }  ,draw opacity=1 ]   (495,81) .. controls (597,132.5) and (536,27.5) .. (654,237.5) ;
\draw  [color={rgb, 255:red, 16; green, 232; blue, 14 }  ,draw opacity=1 ][dash pattern={on 4.5pt off 4.5pt}][line width=0.75]  (527.5,97) .. controls (527.5,83.19) and (538.69,72) .. (552.5,72) .. controls (566.31,72) and (577.5,83.19) .. (577.5,97) .. controls (577.5,110.81) and (566.31,122) .. (552.5,122) .. controls (538.69,122) and (527.5,110.81) .. (527.5,97) -- cycle ;

\draw (122,79.4) node [anchor=north west][inner sep=0.75pt]    {$p$};
\draw (549,79.4) node [anchor=north west][inner sep=0.75pt]    {$p$};
\draw (353,79.4) node [anchor=north west][inner sep=0.75pt]    {$p$};
\draw (122,130) node [anchor=north west][inner sep=0.75pt]   [align=left] { \ };
\draw (68,230) node [anchor=north west][inner sep=0.75pt]   [align=left] {$\rho -maximum  $};
\draw (290,229.4) node [anchor=north west][inner sep=0.75pt]    {$\rho -minimum$};
\draw (495,228.4) node [anchor=north west][inner sep=0.75pt]    {$\rho -\textit{inflectional}  $};
\end{tikzpicture}

\medskip

It follows from the connectedness of the Milnor arcs at infinity that:

\begin{lemma}\label{lem: RhomaxMinInfe}
Let $f: \bR^2 \to \bR$ be a primitive polynomial function.  Then  each Milnor arc at infinity of $f$ has a well defined $\rho$-type of tangency. 
\end{lemma}

\section{ Ordered 
Milnor arcs and $ \mu$-Clusters }\label{sec":Order&Clust}
In \cite{CP} the authors use the polar curve, where the fibres are tangent to vertical lines (compare with the type of tangency defined in \S \ref{Milnor arcs at infinity}), to find regular values which are atypical, (see Definition \ref{Def:Atyicfibres} below).  Their method detects the phenomena of ``vanishing'' and ``cleaving''  that produce atypical fibres. Their approach consists in making appropriate clusters of unbounded  branches of the polar curve.  We remark that the presence of a polar branch does not produce an atypical fiber as shown by the example explained in \cite[Example 3.4]{TZ}.  Another type of clusters, defined in a different context, is used in \cite{HN} to detect atypical fibers.  In this section we define clusters out of Milnor arcs of a primitive polynomial function  (see \S \ref{Milnor arcs at infinity}) and we refer to them as $\mu$-clusters to remark the difference with the clusters of the polar curve defined in \cite{CP}.

\

By definition, the Milnor arcs at infinity do not intersect mutually. It follows that if $C\subset \bR^{2}$ is some large enough circle centred at the origin and such that intersects all Milnor arcs, then $M(f) \cap C$ is a finite set of points $\{p_{1}, \ldots, p_{s}\}$.  We define the following counterclockwise relation between these points\footnote{Note that this is not an order relation.}: we say that ``\emph{$p_{j}$ is the consecutive of $p_{k}$}'', or that ``\emph{$p_{k}$ is the antecedent of $p_{j}$}'',   if and only if  starting from the point $p_{k}$ and moving counterclockwise along the circle $C$ one arrives at the point $p_{j}$  without meeting any other point of the set $M(f) \cap C$.  

We also say that $\{p_{1}, \ldots, p_{k}\}$ is a sequence of consecutive  points of the set $M(f) \cap C$ if and only if $p_{i+1}$ is the consecutive of $p_{i}$ for all $i=1, \ldots, k-1$.   This relation between the points $M(f) \cap C$ on the circle $C$ allows us to define a similar one among the Milnor arcs at infinity, as follows:  

\begin{definition}\cite{MT}(Counterclockwise ordering of Milnor arcs at infinity)\label{Def:ordMilnArc} \ \\
We say that
``\emph{$\gamma_{j}$ is the consecutive of $\gamma_{k}$}'', or that ``\emph{$\gamma_{k}$ is the antecedent of $\gamma_{j}$}'',  if and only if  the point $p_{j} := \gamma_{j}\cap C$ is the consecutive of the point $p_{k} := \gamma_{k}\cap C$. This relation is independent on the size of the circle $C$, provided large enough.  We also say that $\{\gamma_{1}, \ldots, \gamma_{k}\}$ is a sequence of consecutive  Milnor arcs at infinity if and only if $\{p_{1}, \ldots, p_{k}\}$, where $p_{i}:= \gamma_{i}\cap C$, is a sequence of consecutive  points of the set $M(f) \cap C$.
\end{definition}

Let $\gamma \in \mathfrak{M}_{\text{arc}}(f)$ and let  $ \gamma : \left(R_0, \infty \right) \to \bR^2$  be a parametrization  such that $ \lim_{t \to \infty} \| \gamma(t)\| = \infty$. If $ \gamma \cap \Sing f = \emptyset$, then the function $ f (\gamma(t))$ is either strictly increasing or strictly decreasing by Proposition \ref{Prop:RLargEnough} (b). Then we define:  

\begin{definition}\label{Def:In(de)creasing MilnArc}
 Let $ f:\bR^2 \to \bR$ be a polynomial function and let $ \lambda \in \bR \cup \{ \pm \infty\}$  such that $ \lim_{t \to \infty} f(\gamma(t)) = \lambda $. We say that $ \gamma$ is  \textit{increasing to} $ \lambda$ and denote it by $f\stackrel{\gamma }{\nearrow} \lambda$ if $ f(\gamma(t))$ is increasing. Similarly, we say that $ \gamma$ is a \textit{decreasing to} $ \lambda$ and denote it by $f\stackrel{\gamma }{\searrow} \lambda$ if $ f(\gamma(t))$ is decreasing.
 
\end{definition}

\begin{definition}[Clusters of Milnor arcs at infinity]\label{d:cluster}
We call \textit{increasing cluster at $\lambda\in \bR  \cup \{+\infty\}$} a sequence of consecutive Milnor arcs at infinity $\gamma_k, \ldots ,  \gamma_{k+l}$, $l\ge 0$,  such that the condition $f\stackrel{\gamma_i}{\nearrow} \lambda$ holds precisely for all $i= k, \ldots , k+l$ and does not hold for the antecedent of $\gamma_k$ nor for the consecutive of $\gamma_{k+l}$.

Similarly, we define a \textit{decreasing cluster at} $\lambda  \in \bR  \cup \{-\infty\}$ by replacing $\searrow $ instead of $ \nearrow$ in the above definition. 
\end{definition}
A similar definition of Milnor clusters was given in \cite{HN} in  the  setting of surfaces in 
$\bR^{n}$ instead of $\bR^{2}$.

\

Let $C_R$ be the circle centred at the origin with radius $R>R_0$, where $R_0$ denotes the Milnor radius at infinity of the primitive polynomial $f$.  Let $\gamma_{i},  \gamma_{j} \in \mathfrak{M}_{\text{arc}}(f)$ such that $\gamma_j$ is the consecutive of $\gamma_i$ and let $p_i = \gamma_i  \cap C_R$ and $p_j= \gamma_j \cap C_R$.  We denote by $\Gamma^{i,j}_R$ the set of all points in $C_R$ that one meets when moving counterclockwise along $C$   from $p_i$ to $p_j$.  Finally, one defines \textit{the band} between $ \gamma_i$ and $\gamma_j$ (see also \cite[Definition 2.4]{HN}) as: 

\[
]\gamma_i,\gamma_j[ := \bigcup_{R>R_0} \Gamma^{i,j}_{R},
\] and we denote its topological closure in $  \bR^2$ by $  \left[ \gamma_i , \gamma_j  \right]$. 

\medskip

 We denote by $ \gamma(s) $ a parametrization  $ \gamma$ such that $ \lim_{s \to \infty} \| \gamma(s) \| = + \infty$ as in Definition \ref{Def:In(de)creasing MilnArc}. The following lemma is adapted from \cite[Lemma 3.1]{HN}.

\begin{lemma} \label{Lem: estractedfromHN}
Let  $ f: \bR^2 \to \bR$ be a primitive polynomial function,  $ \gamma_i \in  \mathfrak{M}_{\text{arc}}(f)$  increasing to $ \lambda \in \bR$. Then there exists $ \varepsilon>0$ such that for every $ t \in \left] \lambda- \varepsilon , \lambda \right[$ there exists a unique $ s_t > R_0$ such that $ f( \gamma_{i}(s_t)) = t$. Moreover: 

\begin{enumerate}

\rm \item \it  If ${\normalfont \text{Jac}}(f,\rho) >0 $ in $ \left]  \gamma_{i-1}, \gamma_{i} \right[$ and $ {\normalfont \text{Jac}}(f,\rho)<0$ in $ \left]  \gamma_{i}, \gamma_{i+1} \right[$, then there exist $ \delta_{-}$ and $ \delta_{+}$ in $ \left[ R_0 , s_t \right[$ such that the intersection of $ f^{-1}(t)$ with the band $ \left] \gamma_{i-1} , \gamma_i \right[$ \big(resp. $\left] \gamma_{i}, \gamma_{i+1}  \right[  $ \big) is a continuous curve $ \widetilde{h}: \left]  \delta_{-}, s_t\right[ \to \bR^2$ \big(resp. $ \widetilde{h}: \left]  \delta_{+}, s_t\right[ \to \bR^2$\big) with $ \| \gamma_{i}(s)\| = \| \widetilde{h}(s) \|$ for all $ s  \in \left] \delta_{-}, s_t\right[$ \big(resp. $ s  \in \left] \delta_{+}, s_t \right[$\big). 

\smallskip
\rm \item \it If $ {\normalfont \text{Jac}}(f,\rho)<0 $ in $ \left]  \gamma_{i-1}, \gamma_{i} \right[$ and $ {\normalfont \text{Jac}}(f,\rho)>0$ in $ \left]  \gamma_{i}, \gamma_{i+1} \right[$, then there exist $ \delta_{-}$ and $ \delta_{+}$ in $ \left]s_t , \infty \right[$ such that the intersection of $ f^{-1}(t)$ with the band $ \left] \gamma_{i-1} , \gamma_i \right[$ (resp. $\left] \gamma_{i}, \gamma_{i+1}  \right[  $) is a continuous curve $ \widetilde{h}: \left]   s_t, \delta_{-}\right[ \to \bR^2$ \big(resp. $ \widetilde{h}: \left]  s_t, \delta_{+}\right[ \to \bR^2$\big) with $ \| \gamma_{i}(s)\| = \| \widetilde{h}(s) \|$ for all $ s  \in \left]  s_t, \delta_{-} \right[$ \big(resp. $ s  \in \left] s_t, \delta_{+}\right[$\big). 

\smallskip
\rm \item \it  If $ {\normalfont \text{Jac}}(f,\rho) $ has the same sign in $ \left]  \gamma_{i-1}, \gamma_{i} \right[$ and in $ \left]  \gamma_{i}, \gamma_{i+1} \right[$, then there exist $ \delta_{-}$ and $ \delta_{+}$ in $ \left] R_0 , \infty \right[$ such that the intersection of $ f^{-1}(t)$ with the band $ \left] \gamma_{i-1} , \gamma_i \right[$ \big(resp. $\left] \gamma_{i}, \gamma_{i+1}  \right[  $\big) is a continuous curve $ \widetilde{h}: \left]  \delta_{-}, s_t\right[ \to \bR^2$ \big(resp. $ \widetilde{h}: \left]  s_t, \delta_{+}\right[ \to \bR^2$\big) with $ \| \gamma_{i}(s)\| = \| \widetilde{h}(s) \|$ for all $ s  \in \left] \delta_{-}, s_t \right[$ \big(resp. $ s  \in \left]  s_t,\delta_{+} \right[$\big), or a continuous curve $ \widetilde{h}: \left]   s_t, \delta_{-}\right[ \to \bR^2$ \big(resp. $ \widetilde{h}: \left]  \delta_{+}, s_t\right[ \to \bR^2$\big) with $ \| \gamma_{i}(s)\| = \| \widetilde{h}(s) \|$ for all $ s  \in \left]  s_t, \delta_{-} \right[$ \big(resp. $ s  \in \left] \delta_{+}, s_t\right[$\big). 
\end{enumerate} 
\end{lemma}
\smallskip
\begin{remark}\label{Rem:AlternancyRhoType}
From Lemma \ref{Lem: estractedfromHN} neither of the following two situations occurs: 

$(1)$ both $ \gamma_i ,\gamma_{i+1}$  have $\rho$-maximum type. 

$(2)$ both $ \gamma_i ,\gamma_{i+1}$  have $\rho$-minimum type. \end{remark}

\smallskip

\section{Injective function between connected components of the fibres of $f$ and clusters}\label{sec:injectiveFunction}


In this section we show that there exists an injective function between the  set of $ \mu$-clusters with the set of connected components of fibres of the restriction of $f$ outside a compact set, see Theorem \ref{Thm: injFunctAlpha}.  This result plays an important role in the characterization of the atypical values of a primitive polynomial function $f$ as we will show in next section.  For avoiding repetition, we treat only the case of increasing $\mu$-clusters. The treatment for decreasing $\mu$-clusters  is analogous.

\begin{proposition}\label{pro:alphWdefFunc}
Let $f:\bR^2 \to \bR$ be a primitive polynomial function, let $ \mathcal{C}_1 , \ldots , \mathcal{C}_{r}$ be all the increasing $\mu$-clusters associated to a regular value $ \lambda \in \bR $.  Then there exists $\eta^{\ast} > 0$ such that for every $ t \in \left] \lambda - \eta^{\ast} , \lambda \right[$ and every $ i =1, \ldots , r$ there exists a unique connected component $ M_{t,k}$ of $ f^{-1}(t)$ such that $ M_{t,k}$ intersects all Milnor arcs at infinity in $ \mathcal{C}_{i}$. 
\end{proposition}
\begin{proof}
Let $ \mathcal{C}_1 , \ldots , \mathcal{C}_{r}$ be all the increasing $\mu$-clusters associated to a regular value $ \lambda \in \bR $.  Applying the first part of Lemma \ref{Lem: estractedfromHN}  to each Milnor arc at infinity in $ \cup _{i-1}^r \mathcal{C}_i$,  there exist $ \eta^{\ast}>0$ such that for every $ t \in ] \lambda-\eta^{\ast},\lambda [$ the fibre $ f^{-1}(t) $ intersects each $ \gamma \in \cup _{i-1}^r \mathcal{C}_i $ in a unique point.

For each cluster $ \cC_j$, $ j=1, \ldots , r,$  if $ \cC_j$ has a unique Milnor arc  at infinity there is nothing to be proved.  Otherwise, $ \cC_j$ has more than one Milnor arc at infinity and thus we chose $ \gamma_i , \gamma_{i+1}$ two consecutive Milnor arcs in $ \cC_j$ and let $ M_{t,k}$ be the connected component of $ f^{-1}(t)$ which intersect $ \gamma_i$.  Since $M_{t,k}$ is a regular fiber, it intersect the band $[\gamma_i,\gamma_{i+1}]$, and so it intersect $\gamma_{i+1}$ by Lemma \ref{Lem: estractedfromHN}.  Applying the same reasoning  inductively over the Milnor arcs at infinity in $ \cC_j$ we conclude that $ M_{t,k}$ intersect every Milnor arc at infinity in $ \cC_j$. 

The uniqueness follows from the fact that $ f$ restricted to any Milnor arc is strictly monotone, see Proposition \ref{Prop:RLargEnough} (b). 
\end{proof}

Proposition \ref{pro:alphWdefFunc} proves that there is a well defined function  between the set of $ \mu$-clusters and associated to the same regular value $\lambda$ and the set of fibre components of $f$:

\begin{definition}\label{Def:alphaCorresp}
\normalfont Let $f : \bR^2 \rightarrow \bR $ be a primitive polynomial function. Let $\mathcal{C}_{1} , \ldots , \mathcal{C}_{r}$ be all the increasing  $\mu$-clusters
 associated to $\lambda \in \bR$ and let $ \eta^{\ast}>0$ as in Proposition \ref{pro:alphWdefFunc}. Then for every $ t \in \left] \lambda - \eta^{\ast}, \lambda  \right[$ we define the function

\begin{equation*}
\begin{split}
 \alpha'_{t}: \{ \mathcal{C}_i\}_{i=1 ,\ldots,r} &  \rightarrow \{ \text{connected components of } f^{-1}(t)  \}\\
\mathcal{C}_{i}& \mapsto M_{t,j},
\end{split}
\end{equation*} where $M_{t,j}$ is a connected component of $ f^{-1}(t)$ intersecting all Milnor arcs at infinity of $ f $ in $\mathcal{C}_i$ (see Proposition \ref{pro:alphWdefFunc}).   Similarly, one can define this function for the decreasing $\mu$-clusters of $f $ associated to $ \lambda$.
\end{definition}

\smallskip

For any $R>0$ we denote by $f_{R}$ the restriction of $f$ to $ \bR^2 \m \overline{D}_{R}$. Then each fibre $ f_{R}^{-1}(t)$ is a finite union of connected components that we denote by $ F_{t,j}$.  The function $ \alpha'_{t}$ is well-defined by Proposition \ref{pro:alphWdefFunc}.

\begin{theorem}\label{Thm: injFunctAlpha}
Let  $ f : \bR^2 \to \bR$ be a primitive polynomial function,  let $ \mathcal{C}_1  , \ldots , \mathcal{C}_{n}$ be all the increasing $\mu$-clusters of $ f$ associated to $ \lambda \in \bR$. Then there exist $ \eta >0$ and $ R >0 $ large enough such that the function

\begin{equation}\label{e:corrAlph}
\begin{split}
 \alpha_{t}: \{ \mathcal{C}_i\}_{i=1 ,\ldots,n} &  \rightarrow \{ \normalfont\text{connected components of } f_{R}^{-1}(t)  \}\\
\mathcal{C}_{i}& \rightarrow F_{t,j},
\end{split}
\end{equation} is an injective function for all  $ t \in \left] \lambda - \eta , \lambda  \right[  $. 
\end{theorem}
\begin{proof}
By Proposition  \ref{pro:alphWdefFunc} there exists $ \eta ^{\ast} >0 $ such that for every $ t \in ]\lambda - \eta ^{\ast},\lambda [$ we have the well defined function $ \alpha'_t$ as in Definition \ref{Def:alphaCorresp}.  Let $ \beta_{1},\ldots , \beta_s \in   \bR \cup \{ \pm \infty \} \m \{\lambda\} $ such that there exists $ \nu \in \mathfrak{M}_{\text{arc}}(f)$ with $ f_{|\nu} \to \beta_j $ for some $ j=1, \ldots , s$.  Now choose $ \eta>0 $ such that $  \lambda - \eta^{\ast} , \beta_{1} , \ldots , \beta_{s}  \not\in \left] \lambda - \eta, \lambda \right[$  and let us fix $ R>\text{max} \{ R_0, \|p_1\|, \ldots ,\| p_l\| \}$, where the points $ p_1 , \ldots , p_{l} \in \bR^2$ are the elements of the finite set $f^{-1}(\lambda) \cap M(f)\setminus (D_{R_0} \cup \Sing f)$ and $ R_0$ is the Milnor radius at infinity of $f$ as defined before Remark \ref{re:}.  We shall prove that for this choice of $ \eta>0 $ and $ R>0$ the function $\alpha_t$ in \eqref{e:corrAlph} is injective. In order to do it we prove that for any two increasing $ \mu$-clusters $ \cC , \cC'$ associated to $ \lambda$ there is no connected component $ F_{t,k}$ of $ f_{R}^{-1}(t)$ for $ t \in ]\lambda - \eta , \lambda[$ intersecting both $ \mu$-clusters. By contradiction let us assume that there exists such connected component $ F_{t,k}$ intersecting $ \cC $ and $\cC'$ and let us denote by $ \gamma_{1} < \ldots <\gamma_{k_{1}} $  the ordered Milnor arcs at infinity in $ \cC$ and $\gamma_{1}' < \ldots < \gamma_{k_{2}}'$ the ordered Milnor arcs in $ \cC'$. Since the $ \alpha_{t} $ defined for $ f_{R}$ is a well defined function, the connected components $ F_{t,k}$ intersect all $ \gamma_i$ and $ \gamma'_{j}$ for $ i=1,\ldots,k_1$ and $ j = 1, \ldots , k_2$.  Without lost of generality we will assume that $ \gamma_{k_1} < \gamma'_{1}$. In particular $ F_{t,k}$ intersects $\gamma_{k_1}$ and $ \gamma'_{1} $.  Since $ \cC $ and $ \cC'$ are two different $\mu$-clusters we have that there exists $ \xi \in \mathfrak{M}_{\text{arc}}(f) $ such that $ \gamma_{k_1} < \xi < \gamma'_{1} $ and $ \xi$ is not increasing and associated to $ \lambda$. Otherwise we contradict Definition \ref{d:cluster}. Moreover, from the fact that $ F_{t,k}$ is connected and intersects the increasing Milnor arcs at infinity $ \gamma_{k_1}, \gamma'_{1}$, we conclude that $\xi$ is also increasing and thus associated to $ \beta_j \neq \lambda$ for some $ j = 1,\ldots , s$.
 
Since $ t > \lambda - \eta$ and $\beta_j \not\in \left] \lambda - \eta , \lambda \right[ $ we conclude that $\beta_j> \lambda$ and from Proposition \ref{Prop:RLargEnough} (b) there exists $ p \in \xi _{i_s} \setminus D_{R}$ such that $p \in f^{-1}(\lambda)$. This is a contradiction with our choice of $ R$ and thus $F_{t,k}$ does not intersect  intersect $ \gamma_{k1}$ and $ \gamma'_1$. This proves that $ \alpha_t$ is injective and finished the proof.
\end{proof}

\section{Atypical values and $\mu$-clusters}\label{sec:AtipcDetect}

In this section we present the main result of this paper, it is a characterization of the regular atypical values in terms of the parity of the $\mu$-clusters associated to them. Our characterization is based in the detection of the vanishing and splitting phenomena  at infinity (as defined in \cite{DJT}) making use of the correspondence given in Theorem \ref{Thm: injFunctAlpha}.   Let us start by recalling some definitions and statements related to atypical values from \cite{T2,JT,DT,DJT}. 

\begin{definition}\cite{T2}\label{Def:Atyicfibres}
\normalfont 
We say that $f_R$ is a $\mathcal{C} ^{\infty}$ trivial fibration at $\lambda \in \bR$ if there is a neighborhood $I$ of $\lambda \in \bR$ such that the restriction $f_{R}: f_{R}^{-1}(I) \rightarrow I$ is a $\mathcal{C}^{\infty}$ trivial fibration and we say that $ \lambda$ is a \textit{typical value of $f_{R}$}. If $\lambda \in \bR$ does not satisfy this property, then we say that $\lambda$ is an \textit{atypical value} of $ f_{R}$ and that $f_R^{-1}(\lambda)$ is an \textit{atypical fibre}.
\end{definition}


Next definition will present the vanishing and splitting phenomena at infinity.  In order to define the splitting at infinity we follow \cite{DJT}, and thus we first introduce the notion of limit of sets:  let $ \{ M_t \}_{t\in \bR}$ be a family of sets in $ \bR^2$, the limit set of the family $ \{ M_t \}_{t\in \bR}$ when $ t \to \lambda $ denoted by $ \lim _{t\to \lambda} M_t$ is the set of points $ x \in \bR^m$ such that there exists a sequence $ t_k \in \bR $ with  $ t_k \to  \lambda$ and a sequence of points $ x_{k} \in M_{t_k}$ such that $ x_{k} \to x$.      

\medskip

\begin{definition}\label{Def:VanishConnectedComponent}\cite{DJT} Let $ \lambda \in \bR$ such that $ \text{Sing} f^{-1}(\lambda)$ is a compact set. 

\noindent 
(i) One says that $ f$ \textit{has a vanishing at infinity at} $ \lambda$, if either $ \lim_{t \nearrow \lambda} \mathop{\text{max}_j} \mathop{ \text{inf}_{q \in F_{t,j}}} \| q\|= \infty$ or $ \lim_{t \searrow \lambda} \mathop{\text{max}_j} \mathop{ \text{inf}_{q \in F_{t,j}}} \| q\|= \infty$.  Otherwise, we say that $ f$ \textit{has no vanishing at infinity at} $\lambda$ and we denote it shortly by $ \text{NV}(\lambda)$.   

\medskip
\noindent
(ii) One says that $f$ \textit{has a splitting at infinity at $\lambda$ } if there exists $ \eta >0$ and a continuous family of analytic paths $ \phi_{t} : \left[ 0,1  \right] \to f_{R}^{-1}(t) $ for $ t \in \left]  \lambda - \eta, \lambda \right[$ or for $ t \in \left]  \lambda , \lambda+ \eta \right[$, such that:

\smallskip

(1) $ \text{Im}\phi_t  \cap M(f) \neq \emptyset$, with $ \lim _{t \nearrow \lambda}\|q_t\| = \infty$ (or $ \lim _{t \searrow \lambda}\|q_t\| = \infty$, resp.)   for any $ q_t \in \text{Im}\phi_t \cap M(f) $, and

(2) the limit set $\lim_{t \nearrow \lambda} \text{Im}\phi_t$ (or $\lim_{t \searrow \lambda} \text{Im}\phi_t$, resp.) is not connected.

\smallskip

\noindent Otherwise we say that $ f$ \textit{ has no splitting at infinity  at $ \lambda$} and we denote it by NS$(\lambda)$.
\end{definition}

The proof of the main theorem relies on the equivalence between the existence of atypical values with vanishing and splitting phenomena.  An equivalence for the existence of atypical values with phenomena at infinity were proved in \cite[Corollary 4.7]{JT} and thus extended in \cite[Theorem 2.8]{DJT}   to the vanishing and splitting in Definition \ref{Def:VanishConnectedComponent}. 

\begin{proposition}\cite[Theorem 2.8]{DJT}\label{Cor: NV and NS forFrR}
Let $ f^{-1}(\lambda)$ be a fibre with a compact singular set. Then \textit{NV}($\lambda$) and NS($\lambda$) if and only if $ \lambda $ is a typical value of $f_R$.  
\end{proposition}

Next definition distinguishes two kinds of $\mu$-clusters.  

\begin{definition}\label{Def:EvenOdd} Let $\mathcal{C}$ be a $\mu$-cluster associated to $\lambda$.  We say that $ \mathcal{C}$ is odd (or even) if the number of Milnor arcs at infinity with $ \rho$-maximum type in $ \mathcal{C}$ added by the number of Milnor arcs at infinity with $ \rho$-minimum type in $\mathcal{C}$ is an odd (resp. even) number.  \end{definition}

We are ready to prove the main theorem of this section.

\begin{theorem}\label{The:Atyp and Cluster}
Let $ f : \bR^2 \to \bR $ be a primitive polynomial function. A regular value $\lambda \in \bR$ is an atypical value of $f_{R}$ if and only if there exists an odd $\mu$-cluster associated to $\lambda$. 

\end{theorem}

\begin{proof}
Let us prove the result for increasing $\mu$-clusters and the analogous reasoning apply to decreasing $\mu$-clusters. Let $\mathcal{C}$ be an increasing $\mu$-cluster associated to $\lambda \in \bR$. Choose $ R >0$ and $ \eta >0$ as in Theorem \ref{Thm: injFunctAlpha} such that the function $ \alpha_t$ is injective for every $ t \in \left] \lambda - \eta , \lambda \right[ $ and by choosing a smaller $\eta$ if necessary we assume that the interval $ ] \lambda- \eta, \lambda [$ contains only regular values of $f$. Let $\gamma_i < \cdots < \gamma_j $ be the consecutive sequence of Milnor arcs at infinity in $\mathcal{C}$. 

\medskip
 If $\mathcal{C}$ is odd, then let $\gamma_{s_1} < \cdots < \gamma_{s_{2k+1}}$ be all the Milnor arcs at infinity in $ \mathcal{C}$ with $\rho$-extrema type.  First let us treat the case when $ \gamma_{s_1}$ has $\rho$-minimum type, and thus  $\gamma_{s_{2k+1}}$ has also a $ \rho$-minimum type since $ \cC$ is odd and Remark \ref{Rem:AlternancyRhoType}. Therefore there exists a global minimum of $\rho$ restricted to $\alpha_t(\mathcal{C})$, that we denote by $\inf_{x \in \alpha_{t}(\mathcal{C})}\|x\|  $. Since $\mathcal{C}$ is increasing and associated to $ \lambda$,  one has that $\lim _{t \nearrow \lambda} \inf_{x \in \alpha_{t}(\mathcal{C})} \|x\|=\infty. $ This proves that  $f$ has a  vanishing at infinity at $ \lambda$ by Definition  \ref{Def:VanishConnectedComponent}.

Now let us treat when $\gamma_{s_1}$ has a $\rho$-maximum type.  By Theorem \ref{Thm: injFunctAlpha}  $ \alpha_t(\cC) $ intersects $ \gamma_{s_1}$ for all values $ t \in ]\lambda-\eta,\lambda[$ and all those values are regular, thus,  from the Implicit Function Theorem we obtain a family of continuous  analytic paths $ \ \phi_t : [0,1] \to \alpha_t(\cC)$ such that $\text{Im}\phi_t \cap \gamma_{s_1} = q_t  $ and $ \lim_{t \nearrow \lambda} \| q_t \|=\infty $.  Since $ \gamma_{s_1}$ has $\rho$-maximum, also $\gamma_{s_{2k+1}} $ has $ \rho$-maximum by Remark \ref{Rem:AlternancyRhoType}. Thus the connected component $ \alpha_t(\cC)$ is bounded,  and $ \alpha_t(\cC) \cap ]\gamma_{s_1-1},\gamma_{s_1}[ \neq \emptyset$ and $ \alpha_t(\cC) \cap ]\gamma_{s_{2k+1}},\gamma_{s_{2k+2}}[ \neq \emptyset$ for all $t \in ] \lambda - \eta , \lambda [$.  This implies that $ \lim_{t \nearrow \lambda} \alpha_t(\cC) \cap ]\gamma_{s_1-1},\gamma_{s_1}[  \neq \emptyset$ and $ \lim_{t \nearrow \lambda}\alpha_t(\cC) \cap ]\gamma_{s_{2k+1}},\gamma_{s_{2k+2}}[ \neq \emptyset$.  On the other hand,  if we denote by $ q^i_{t} $ the element in $ \gamma_{s_i} \cap \alpha_t(\cC)$,  then we have by Proposition \ref{Prop:RLargEnough}  (a) that $ \lim_{t \nearrow \lambda} \| q_t^i \| = \infty$ for $ i= 1, \ldots , 2k+1$.  Consequently, $ \lim_{t \nearrow \lambda} \text{Im}\phi_t $ has at least two connected components, one in $ ] \gamma_{s_1-1}, \gamma_{s_{1}} [$ and another one in $ ]\gamma_{s_{2k+2}, \gamma_{s_{2k+2}}}[$.  This proves that $ \alpha_t 
(\cC)$ has a splitting at infinity.    

Now we prove that if all $ \mu$-clusters associated to $\lambda$ are even, then $f$ has no vanishing and no splitting.  It is enough to show that any even $\mu$-cluster has neither vanishing nor splitting.  Since $ \cC$ is even, let $ \gamma_{s_1} < \cdots < \gamma_{s_{2k}}$ be all the Milnor arcs at infinity $\mathcal{C}$ with $\rho$-extrema type.  From the alternation of the $\rho$-type in Remark \ref{Rem:AlternancyRhoType} $ \gamma_{s_1}$ and $ \gamma_{s_{2k}}$ have different $\rho$-type. Let us assume that $ \gamma_{s_1}$ has $ \rho$-maximum type and thus $ \gamma_{s_{2k}}$ has $\rho$-minimum type.   Hence the intersection $ \alpha_t(\cC)  \cap ] \gamma_{s_{1}-1} , \gamma_{s_{1}} [$ is not empty and is contained in the disk $ D_{\|q_t\|} $, where $ q_t= \gamma_{s_1} \cap \alpha_{t}(\cC)$ for all $ t \in ] \lambda - \eta, \lambda [$.  This implies that  $ \lim_{t \nearrow \lambda}  \alpha_{t}(\cC)\neq \emptyset$. Consequently $ \alpha_t(\cC)$ has no vanishing at infinity.  It is left to be proved that $ \alpha_{t}(\cC)$ has no splitting at infinity. First let us notice that the set  $ \lim_{t \nearrow \lambda} \alpha_{t}(\cC) \cap ] \gamma_{s_{2k}},\gamma_{s_{2k+1}} [ = \emptyset $. Indeed,  since $ \gamma_{s_{2k}}$ has a $\rho$-minimum type,  the set $ \alpha_{t}(\cC) \cap ] \gamma_{s_{2k}}, \gamma_{s_{2k+1}} [$ is contained in the exterior of the disk $ D_{\| q_t^{2k}\|} $, where $ q_{t}^{2k} = \alpha_{t}(\cC) \cap \gamma_{2k}$. Consequently, $ \lim_{t \nearrow \lambda}  \alpha_{t}(\cC)\cap ] \gamma_{s_{2k}},\gamma_{s_{2k+1}} [ =\emptyset $, since $ \lim_{t \nearrow \lambda} \| q_{t}^{2k} \| = \infty$.  With this observation we have that for every family of continuous paths $\phi_{t}:[ 0,1 ] \to \alpha_{t}(\cC)$ as in Definition \ref{Def:VanishConnectedComponent} the limit set  $ \lim_{t \nearrow \lambda} \text{Im}\phi_{t} $ is contained in $\lim_{t \nearrow \lambda} \alpha_{t}(\cC)  \subset ] \gamma_{s_{1}-1} , \gamma_{s_{1}}[$, and thus $ \text{Im}\phi_{t}$ is connected.  With this we proved that $\alpha_{t}(\cC)$ has no splitting at infinity.  Similarly we obtain the same conclusion if  $ \gamma_{s_{1}}  $ is $ \rho$-minimum and thus this prove that if all the $\mu$-clusters associated to $ \lambda$ are even, then NV$(\lambda)$ and NS$(\lambda)$.

From the above proof we conclude that all $\mu$-clusters associated to $\lambda$ are even if and only if NV($\lambda$) and NS($\lambda$) holds. Applying Proposition \ref{Cor: NV and NS forFrR} we conclude that $f$ has no atypical values and this ends the proof.
\end{proof}

\section{Algorithmic aspects}\label{sec:AlgotithAspec}

In the following we present the algorithmic aspects for the detection of the regular atypical values at infinity. 

\subsection{Milnor radius at infinity }\label{subsec:MilnorRadius}
Consider $h(x,y)$ the reduced polynomial defining $ M(f)_{\text{red}}$ as in Definition \ref{Def: MM }.  The set $ \mu(M(f))$ is the solutions of the system 
 \begin{equation}\label{Eq:systemMilRad}
\begin{cases}
h(x,y)=0 ,\\
x h_{y}(x,y) - y h_{x}(x,y)=0.
\end{cases} 
 \end{equation}
 
 By Proposition \ref{Prop:RLargEnough} it is enough to choose $ R > {\normalfont\text{max}}_{q \in \mu(M(f))}\{\|q\|\}$ and this choice can be done effectively: by Proposition \ref{Prop:DetecCircleInMM} 
  the set $ \mu(M(f))$ is a union of finitely many points and circles centred at the origin. Moreover, the circles $ \{ x^2 + y^2 -r^2=0 \}$ contained in $ M(f)$ are identified with the factors $ x^2 + y^2 -r^2 $ occurring in the irreducible decomposition of $ h(x,y) $.     By choosing all such factors $ x^2+y^2-r_i^2 $, $i = 1, \ldots , l_1$  and the isolated solutions $ q_i$, $ i = 1, \ldots ,l_2$ of the system (\ref{Eq:systemMilRad}), we take
   \[
R > \text{max}\{ r_1, \ldots , r_{l_1}, \| q_1\| \ldots , \|q_{l_2} \| \}.  
  \]

\subsection{ Regular finite values associated to Milnor arcs. }\label{subse:findValues}

Consider $ \mathbb{P}^2 = \bR^2  \cup \mathbb{P}^1$ the compactification of $\bR^2$ where $\mathbb{P}^1 = L^\infty$ denotes the line at infinity. Let $\overline{M(f)}$  the projective closure of $M(f)$ in $\mathbb{P}^2$. We are interested in the intersection points between $ \overline{M(f)}$ and the line $ L^{\infty}$. By Proposition \ref{Prop:StrenDTLemm} $(a)$, $ M(f)$ is unbounded and thus $ \overline{M(f)}\cap  L^{\infty} $ is non-empty with finitely many points in $ \mathbb{P}^2$. The set $ \overline{M(f)}\cap L^{\infty}$ can be determined by homogenizing with respect to  the variable $z$ the polynomial $\text{Jac}(f,\rho)$ and thus finding its intersection with $z=0$ in $\mathbb{P}^2$. By Theorem \ref{The:Atyp and Cluster}, if $ \lambda$ is a regular atypical value of $f$, then there is at least one Milnor arcs  at infinity $ \gamma$ of $f$ such that either $f\stackrel{\gamma }{\nearrow} \lambda$ or $f\stackrel{\gamma }{\searrow} \lambda$.

\medskip
In what follows we make a brief summary of the effective algorithm presented in \cite{DJT} which detect the Milnor arcs $ \gamma $ such that the restrictions $f_{|\gamma}$ are bounded. For each point $p \in \overline{M(f)} \cap L^{\infty} \subset \mathbb{P}^2$, there exist Milnor arcs at infinity  $\gamma_{i_1} , \ldots , \gamma_{i_r}$ such that $ p$ is in their closures in $\mathbb{P}^2$.  By a linear change of coordinates if necessary we may assume that $ p = \left[ 0 :1:0 \right] \in \overline{M(f)} \cap L^{\infty}$, and consider the chart $\{ y \neq 0 \} $ of $ \mathbb{P}^2$ with local coordinates $ (x,z)$. Set $ \widehat{f}(x,z) = \widetilde{f}(x,1,z) $ and $ \widehat{h}(x,z)= \widetilde{h}(x,1,z) $, where $ \widetilde{f}(x,y,z) $ and $\widetilde{h}(x,y,z) $ denote the homogenization with respect to the variable $z$ of $ f(x,y) $ and $ h(x,y)$, respectively. The problem of finding the values $\lambda \in \bR$ associated with the Milnor arcs at infinity $ \gamma_{i_1} , \ldots , \gamma_{i_r} $  is equivalent to find all limits

\begin{equation}\label{Prob: limits}
\lim \frac{ \widehat{f}(x,z)}{z^d}  {\normalfont\text{   for   }} (x,z) \to (0,0) \text{ and } \widehat{h}(x,z)=0.
\end{equation}

Passing to complex variables and considering $ \widehat{f}(x,z)$ and $ \widehat{h}(x,z)$ as holomorphic germs at the origin $(0,0)$ of the chart centred at $p$, the authors in \cite{DJT} applied the Newton-Puiseux algorithm to find a Puiseux parametrization of the  Milnor arcs at infinity $ \gamma_{i_{1}}, \ldots \gamma_{i_{r}} $ of the form $ z = T^n, x= \sum_{1 \leq j} a_j T^j$ where the series may be infinite. The limits in (\ref{Prob: limits}) can be determined by truncating the Newton-Puiseux series to a finite number because the following equality

\begin{equation*}\label{Prob:TruncatedN-Puiseux}
\lim_{T \to 0} \frac{\widehat{f}(\sum_{1 \leq j} a_j T^j, T^n)}{T^{dn}}= \lim_{T \to 0} \frac{\widehat{f}(\sum_{1 \leq j \leq dn} a_j T^j, T^n)}{T^{dn}}=\lambda.
\end{equation*}

\subsection{Detecting $\mu$-clusters associated to $ \lambda$.} \label{Sec:Detecting}

Once we have found the values $ \lambda \in \bR$ as in  \ref{subse:findValues} we shall detect the $\mu$-clusters associated to $ \lambda$.  From Definition \ref{Def:MilArcInft} at every point in  a Milnor arc at infinity $\gamma$ the gradient vector of $f$ is parallel to the position vector and from Proposition  \ref{Prop:RLargEnough} (b) the restriction $ f_{|\gamma}$ is strictly monotone.  Notice that our choice of $ R$ in \ref{subsec:MilnorRadius} is such that $ R>R_0$ where $R_0$ denotes the Milnor radius at infinity of $f$ (see Definition \ref{Def:MilArcInft}) and thus by Proposition \ref{Prop:RLargEnough} we shall test if each Milnor arc at infinity $\gamma$ is increasing or decreasing by the following criteria: let $q \in \gamma \cap C_R$ and let $ \text{sgn} : \bR \to \{ \pm ,0\}$ be the sign function. With this we have: 

\smallskip

\noindent (i) if $ \text{sgn}\langle \text{grad} f(q) , q \rangle = +1$, then $ \gamma$ is an increasing Milnor arc at infinity,
\smallskip

\noindent (ii) if $ \text{sgn}\langle \text{grad} f(q) , q \rangle = -1$, then $ \gamma$ is an decreasing Milnor arc at infinity,

\smallskip

\noindent (iii) if $ \langle \text{grad} f(q) , q \rangle = 0$, then by Proposition \ref{Prop:RLargEnough} and Definition \ref{Def:MilArcInft}, $ \gamma \subset \text{Sing}f$.  In such case we have that $\lambda $ is an atypical value by Definition \ref{Def:Atyicfibres}. 

\smallskip

For the cases (i), (ii) we have that for a fixed $ \lambda$  regular obtained in \ref{subse:findValues}  we classify the Milnor arcs at infinity $ \gamma_{i_{1}}, \ldots , \gamma_{i_{r}}$  associated to $\lambda
$ in two sets: $ \mathcal{I}_{\lambda}$ and $\mathcal{D}_{\lambda}$ which are the sets of Milnor arcs at infinity which are increasing and decreasing respectively and associated to the value $ \lambda$. 

In order to determine the $ \mu$-clusters associated to each $ \lambda$ we choose all sequences $ \gamma_k < \gamma_{k+1}  < \dots  < \gamma_l $ of consecutive Milnor arcs such that $ \gamma_k , \gamma_{k+1} , \ldots , \gamma_l  \in \mathcal{I}_{\lambda} $ (resp. $ \in  \mathcal{D}_{\lambda}$) and   $ \gamma_{l+1} , \gamma_{k-1} \not\in \mathcal{I}_\lambda$ (resp. $ \not\in  \mathcal{D}_{\lambda}$). By Definition \ref{d:cluster}  each such sequence $ \gamma_k < \gamma_{k+1}  < \dots  < \gamma_l $ in $ \mathcal{I}_{\lambda}$ (resp. in $ \mathcal{D}_{\lambda} $) defines an increasing (resp. a decreasing) $\mu $-clusters associated to $ \lambda$. 

\subsection{Detecting regular atypical values of $f$.}\label{subsec:detecAtyp}

After Section \ref{Sec:Detecting} we have identified the $ \mu$-clusters associated to each $ \lambda \in \bR $ as in Section \ref{subse:findValues}.  Thus, applying Lemma \ref{Lem: estractedfromHN} one may identify the $\rho$-type of tangency of each Milnor arc at infinity on each $\mu$-cluster.    Using Definition \ref{Def:EvenOdd} we can classify each cluster associated to $ \lambda$ by its parity.  If there exists an odd $\mu$-cluster $ \cC=\{ \gamma_k , \gamma_{k+1},\ldots , \gamma_l \}$  we have the following criteria:

\smallskip
\noindent  (i) If $\gamma_k$ is has $\rho$-maximum type of tangency, then $ \alpha_t(\cC)$ is splitting at infinity. 

\smallskip
\noindent  (ii) If $\gamma_k$ is has $\rho$-minimum type of tangency, then $ \alpha_t(\cC)$ is vanishing at infinity. 

\smallskip
\noindent
In both cases $ \lambda$ is a regular atypical value.  On the other hand,  if all $\mu$-cluster associated $\lambda$ are even, then $\lambda$ is a typical value.





\section{ Examples}\label{sec:Examples}

In this section the algorithmic aspects presented in Section \ref{sec:AlgotithAspec} are applied in two examples presented in \cite{TZ}  in order to  detect  their atypical values.   Example \ref{Ex1} shows that it is possible to have both vanishing and splitting phenomena at infinity at the same value $\lambda$.  Example \ref{Ex2} shows that it may exists only even $\mu$-clusters associated to a regular value, and thus the value is not atypical by Theorem \ref{The:Atyp and Cluster}.
For presenting both examples the software \textit{Mathematica} has been used to make the computations.

\begin{example}\label{Ex1}
Consider

$$  f(x,y) =
 x^2 y^3(y^2-25)^2  + 2 xy (y^2-25)(y+25) - (y^4 + y^3 -50 y^2 - 51 y +575
 ). $$ 

We show that the regular value $\lambda = 0$ of $f$ is atypical. Moreover we show that there are two vanishing components and two splitting components at the regular value $0$.  

\smallskip

Following \ref{subsec:MilnorRadius} we first find the number $R>0$ such that the disk $ D_R$ centred at the origin of radius $R$ contains in its interior the set $ \mu(M(f))$: after decomposing $ h(x,y):= \Jac(f,\rho)$ as a product of irreducible polynomials we conclude that $h$ is an irreducible polynomial and it has no irreducible components of the form $ x^2+y^2-r^2 $.  It follows from Proposition \ref{Prop:DetecCircleInMM} that $ M(f)$ does not contain circle components, and thus $ \mu(M(f))$ is the union of finitely many points by Proposition \ref{Prop:MMCompact}.  By Definition \ref{Def: MM } the set $ \mu(M(f))$ is  algebraic and it is defined by the equations $ h=0$ and $ y h_x(x,y)- x h_y(x,y)=0$.  After finding the points in $ \mu(M(f))$ and their distance to the origin,  as in \ref{subsec:MilnorRadius} one concludes that $ R=10$, and thus $ \mu(M(f)) \subset D_{10}$.

\smallskip

By Proposition \ref{Prop:RLargEnough} all circles centred at the origin and containing in its interior $ \mu(M(f))$ are transverse to the Milnor arcs at infinity of $f$. Therefore, the circle  $C_{10}$ centred at the origin of radius $R=10$ intersect all the Milnor arcs at infinity transversely, since $ \mu(M(f)) \subset D_{10} $.  Hence the number of Milnor arcs coincides with the number of solutions of the equations $ h(x,y) = 0 , x^2+y^2 -10^2=0$, and thus, after solving this system, we conclude that there are sixteen solution points:  $ p_1, \ldots , p_{16}$, where order is giving by the angle defined by each  $p_i$ with the positive $x$ axis.   By Proposition \ref{Prop:RLargEnough}  (a) and Definition \ref{Def:ordMilnArc} there are sixteen ordered Milnor arcs at infinity  $ \gamma_1, \ldots , \gamma_{16}$ with the order induced by the points $ p_1, \ldots , p_{16} $ in their intersection with $ C_{10}$.

\smallskip

Let us follow \ref{subse:findValues} to  find the regular values associated to each Milnor arc: the set $ \overline{M(f)} \cap L^{\infty}$ is the set of points $ p=[a:b:0] \in L^{\infty}$ such that $\widetilde{h}(a,b,0)=0$, where $ \widetilde{h}$ denotes the homogenization with respect to the variable $ z$ of $h$.  By computing its solutions we have that $ \overline{M(f)}\cap L^{\infty}=\{ \left[0:1:0\right], [1:0:0],[\sqrt{2}: \sqrt{7}:0],[-\sqrt{2}: \sqrt{7}:0] \}  $.  By applying the truncated Newton-Puiseux process, presented in \cite{DJT} and described in \ref{subse:findValues}, to each point in $ \overline{M(f)}\cap L^{\infty} $ we have:

\[ 0=\lim_{t \to \infty} f(\gamma_3(t)) = \lim_{t \to \infty} f(\gamma_7(t)) = \lim_{t \to \infty} f(\gamma_{10}(t)) = \lim_{t \to \infty} f(\gamma_{14}(t)), \]and all limits over the other Milnor arcs tend either to $ + \infty$ or $ -\infty$. Therefore $ \lambda=0$ is the unique finite value associated to the Milnor arcs at infinity of $f$. 

\smallskip

 Let us detect the $ \mu$-cluster associated to $0$:  following  \ref{Sec:Detecting} we obtain that 
 
 \begin{equation}\label{e:critIncDecr} 
\begin{matrix}
\text{sgn}\langle p_3 ,  \mathop{\text{grad}}f(p_3) \rangle =+1 ,& \text{sgn}\langle p_7 ,  \mathop{\text{grad}}f(p_7) \rangle =+1 ,\\
\text{sgn}\langle p_{10} ,  \mathop{\text{grad}}f(p_{10}) \rangle =+1 ,& \text{sgn}\langle p_{14},  \mathop{\text{grad}}f(p_{14}) \rangle =+1.
\end{matrix}
 \end{equation}

Since for every point in a Milnor arc the  gradient vector at such point is a non zero multiple scalar of the vector position,  one concludes from \eqref{e:critIncDecr} that the set of increasing Milnor arcs associated to $0$ is $\mathcal{I}_{0}=\{\gamma_3,\gamma_7,\gamma_{10},\gamma_{16}\} $ and the set $ \mathcal{D}_0$ of decreasing Milnor arcs associated to $0$ is empty.   Since there is no Milnor arcs in $ \mathcal{I}_0$ that are consecutive and $ \mathcal{D}_0$ is empty,  the $\mu$-clusters associated to $0$ are:   $ \mathcal{C}_{1} := \{ \gamma_3  \} ,  \mathcal{C}_{2} := \{ \gamma_7  \},  \mathcal{C}_{3} := \{ \gamma_{10}  \}, \mathcal{C}_4 : = \{ \gamma_{14}  \}  $ by Definition \ref{d:cluster},  and all of them are increasing.

By Lemma \ref{Lem: estractedfromHN} one concludes that $ \gamma_3 , \gamma_7 $ have $ \rho$-maximum type of tangency and $ \gamma_{10},\gamma_{14}$ have $ \rho$-minimum type of tangency.  Therefore $ \cC_1 ,\cC_2$ have splitting components and $ \cC_3,\cC_4$ have vanishing components as the proof of Theorem \ref{The:Atyp and Cluster}  shows.  Hence we conclude by Theorem \ref{The:Atyp and Cluster} that 
$ \lambda=0$ is an atypical value at infinity of $ f$ since all $ \mu$-clusters associated to $\lambda=0$ are odd.

\end{example}

\begin{example}\label{Ex2}
 Consider
\[f(x,y) = 2 x^2 y^3 - 9 x y^2 + 12 y. \] 

In \cite{TZ} the authors proved that $ \mathcal{B}_f  = \emptyset$.  We use our algorithm to show that there are only two $\mu$-clusters associated to $\lambda=0$ and neither of them has vanishing components or splitting components.  The polynomial
\[
h(x,y)= \Jac(f,\rho)=12 x - 18 x^2 y + 6 x^3 y^2 + 9 y^3 - 4 x y^4,
\] defining $M(f)$ is an irreducible polynomial different than $ x^2+y^2-r^2$ for every $ r\in \bR$, and thus there are no circle components in $ \mu(M(f))$ by Proposition \ref{Prop:DetecCircleInMM}.  By Proposition \ref{Prop:MMCompact} one concludes that $\mu( M(f))$ is the union of finitely many points. Following \ref{subsec:MilnorRadius} let us find a radius $ R>0$ such that the disk $D_R$ centred at the origin of radius $R$ contains in its interior the set $ \mu(M(f))$.  By Definition \ref{Def: MM } the set $ \mu(M(f))$ is the common solution of the equations $h=$ and $xh_y-yh_x=0$.  After solving this system and computing the distances to the origin of its finitely many solutions, we obtain that $ \mu(M(f))$ is contained in the interior of the disk $ D_3$.

\smallskip

By Proposition \ref{Prop:RLargEnough} (a) all circles containing in its interior the set $ \mu(M(f))$ are transverse to all Milnor arcs of $f$.  Then the circle $C_{3}$ centred at the origin of radius $R=3$ intersects transversely all the Milnor arcs at infinity of $f$, since $ \mu(M(f)) \subset D_{3} $.   The system of equations $ h(x,y)= 0 , x^2+y^2 -3^2=0$ has ten solution points $ p_1, \ldots , p_{10}$ ordered by their positive angle measured from the positive $x$ axis.  It follows from Definition \ref{Def:ordMilnArc} that $ f$ has ten Milnor arcs at infinity $ \gamma_1, \ldots , \gamma_{10}$ with the ordered induced by the points $ p_1, \ldots , p_{10} $.

\smallskip

The set $ \overline{M(f)} \cap L^{\infty}$ contains the points $[a:b:0] \in L^{\infty}$ such that $ \widetilde{h}(a,b,0)=0$ where $ \widetilde{h}$ denotes the homogenization with respect to the variable $z$ of $h$.     Hence $ \overline{M(f)}\cap L^{\infty}=\{ [ 0:1:0  ],[1:0:0 ],[ \sqrt{2}: \sqrt{3}:0 ],[-\sqrt{2},\sqrt{3}:0 ] \} $.  By applying the truncated Newton-Puiseux process, as in \cite{DJT} and in \ref{subse:findValues},  to each point in $\overline{M(f)} \cap L^{\infty} $ one has: 

\[ 0=\lim_{t \to \infty} f(\gamma_1(t)) = \lim_{t \to \infty} f(\gamma_2(t)) = \lim_{t \to \infty} f(\gamma_{6}(t)) = \lim_{t \to \infty} f(\gamma_{7}(t)), \]and the limits over the other Milnor arcs tend to either $ +\infty$  or  $ - \infty$. This implies that $\lambda=0$ is the unique finite value associated to the Milnor arcs at infinity of $f$.
   
\smallskip

On the other hand, we have that

\begin{equation}\label{e:2critIncDecr}
\begin{matrix}
\text{sgn}\langle p_1 ,  \mathop{\text{grad}}f(p_1) \rangle =-1 ,& \text{sgn}\langle p_2 ,  \mathop{\text{grad}}f(p_2) \rangle =-1 ,\\
\text{sgn}\langle p_{6} ,  \mathop{\text{grad}}f(p_{6}) \rangle =+1 ,& \text{sgn}\langle p_{7},  \mathop{\text{grad}}f(p_{7}) \rangle =+1.
\end{matrix}
\end{equation}

Since at each point $p_1,p_2,p_6,p_7$ the gradient vector of $f$ is a non zero scalar multiple of the position vector,  one concludes from \eqref{e:2critIncDecr} that the set of increasing Milnor arcs associated to $0$ is  $\mathcal{I}_{0}=\{\gamma_6,\gamma_7\}$ and the set of decreasing Milnor arcs associated to $0$ is $ \mathcal{D}_{0}=\{ \gamma_1,\gamma_2 \}$.

 Since $ \gamma_6, \gamma_7$ are consecutive and $ \gamma_1,\gamma_2$ are consecutive,  it follows from Definition \ref{d:cluster} that there are only two $\mu$-clusters associated to $ \lambda=0$.  Therefore by Theorem \ref{The:Atyp and Cluster} $ \lambda=0$ is not an atypical value at infinity of $ f$ since 
the only two $\mu$-clusters associated to it are even. Moreover,  $\mathcal{B}_f =\emptyset$ since there are no other Milnor arcs at infinity of $f$ associated to another finite value.

\end{example}

\end{document}